\documentclass{IEEEtran}
\usepackage{comment}
\usepackage{cite}
\usepackage{amsmath,amssymb,amsfonts,hyperref}
\usepackage{algorithmic}
\usepackage{graphicx}
\usepackage{textcomp}
\usepackage{subcaption}
\usepackage{listings}
\usepackage{xcolor}
\usepackage{amsthm}
\def\BibTeX{{\rm B\kern-.05em{\sc i\kern-.025em b}\kern-.08em
		T\kern-.1667em\lower.7ex\hbox{E}\kern-.125emX}}
\begin{document}
	%%%%%%%%%%%%%%%%%%%%%%%%%%%%%%%%%
	\newtheorem{definition}{ Definition}
	\newtheorem{theorem}{ Theorem}
	\newtheorem{lemma}{ Lemma}
	\newtheorem{proposition}{ Proposition}
	\newtheorem{corollary}{ Corollary}
	\newtheorem{remark}{ Remark}
	%%%%%%%%%
	\title{A Periodic Dichotomy in Linear Control Theory}
	\author{Shichao Ye,  Xingwu Zeng, Can Zhang
		\thanks{This work was partially supported  by the National Natural Science Foundation of China under grant 12422118.}
		\thanks{Shichao Ye is with the School of Mathematics and Statistics, Wuhan University, Wuhan 430072, China (e-mail: 2020312011154@whu.edu.cn).}
		\thanks{Xingwu Zeng is with the School of Mathematics and Statistics, Wuhan University, Wuhan 430072, China (e-mail: xingwuzeng@whu.edu.cn).}
		\thanks{Can Zhang is with the School of Mathematics and Statistics, Wuhan University, Wuhan 430072, China (e-mail: canzhang@whu.edu.cn).}}

	\maketitle
	
	\begin{abstract}
		In this paper, we construct a periodic dichotomy transformation using solutions of periodic Riccati and Lyapunov equations. 
		As an application of this transformation, we provide an explicit representation of the optimal extremal for periodic linear quadratic optimal control problems.
		Specifically, we establish a complete characterization of the optimal extremal under suitable exponential stabilizability and detectability assumptions.
	\end{abstract}
	
	
	\section{Introduction}
	\label{sec:introduction}
	A dichotomy transformation is a transformation that decouples a system possessing an ``exponential dichotomy" into two independent subsystems: one exponentially stable in forward time and the other exponentially stable in backward time (see, e.g., \cite{WK}). 
	The authors of \cite{WK} constructed such a transformation using the positive and negative definite solutions of an associated Riccati equation (see \cite[Lemma 3]{WK}).
	The dichotomy transformation technique has been widely applied in the numerical analysis of optimal control problems (see, e.g., \cite{GG, RM}).
	
	In \cite{GG}, an efficient numerical method was developed for solving singularly perturbed matrix differential Riccati equations under stabilizability and observability conditions.
	For the periodic case, several mature approaches for solving the periodic differential Riccati equation (PRDE) have emerged (\cite{GJKSV10,V}).
	Among them, the periodic-generator approach first computes the symplectic monodromy matrix of the Hamiltonian system, reduces it to an ordered real Schur form, and then constructs a linear matrix differential equation whose solution yields the solution of PRDE.
	However, this approach can be numerically unreliable as it requires integrating ODEs with unstable underlying dynamics, leading to error amplification over long periods.
	To address this difficulty, multiple-shooting algorithms subdivide the period into shorter intervals, integrate them separately, and match the segments through suitable boundary conditions to obtain the periodic stabilizing solution.
	Most recently, a convex-optimization reformulation (see, e.g., \cite{GJKSV10}) parameterizes the solution via a truncated Fourier series and recasts the PRDE as a semidefinite program with linear-matrix-inequality constraints. 
	By avoiding the computation of the monodromy matrix, this method scales linearly with respect to the truncation order rather than the period length, thereby providing a scalable framework for large-scale periodic feedback design.
	
	Beyond numerical applications, the dichotomy transformation has proven to be a powerful tool in the analysis of optimal control problems, particularly for investigating the turnpike property (see, e.g., \cite{WK, TZZ, TZ}).
	The underlying mechanism is that the dichotomy structure of the associated Hamiltonian system separates stable and unstable modes, thereby providing a precise description of the long-time behavior of extremals.
	Early works, such as \cite{RM}, employed a dichotomic basis to decompose the Hamiltonian vector field into stable and unstable components, thereby constructing approximations of the optimal solution.
	In \cite{TZ}, a dichotomy transformation was used to analyze the Hamiltonian structure of the linearized extremal system, which plays a central role in establishing the turnpike property for general finite-dimensional nonlinear control systems under the Kalman rank condition.
	Subsequently, \cite{TZZ} extended the turnpike analysis to nonlinear control systems with unbounded operators, including examples arising from partial differential equations.
	Building on these developments, \cite{TZZ2} established an exponential periodic turnpike property for infinite-dimensional linear–quadratic (LQ) problems with periodic coefficients.
	It is worth mentioning that, in the time-invariant LQ setting, \cite[Lemma 1]{TZZ} introduced a dichotomy transformation based on algebraic Riccati and Lyapunov equations, which decouples the Hamiltonian system derived from the Pontryagin maximum principle and forms the structural core of the corresponding turnpike analysis.
	
	Motivated by \cite[Lemma 1]{TZZ}, we establish in this paper a periodic dichotomy transformation constructed from the solutions of periodic Riccati and Lyapunov differential equations under suitable periodic exponential stabilizability and detectability assumptions.
	This work can be viewed as a natural periodic extension of the time-invariant dichotomy transformation in \cite[Lemma 1]{TZZ}.
	
	As an application, we provide an explicit representation of the optimal extremal pair for periodic LQ optimal control problems.
	Moreover, the periodic dichotomy transformation enables us to establish an exponential turnpike property for finite-dimensional LQ systems with periodic coefficients.
	More precisely, for a sufficiently large time horizon, the optimal solution remains exponentially close to a specific periodic orbit during most of the time interval, excluding the transient arcs near the initial and final times.
	Hence, the present work provides a structural explanation of the periodic turnpike phenomenon through the dichotomy transformation framework.
	
	In \cite{Xu21}, periodic exponential stabilizability of a periodic linear control system was shown to be equivalent to a detectability inequality in a Hilbert space setting with bounded control operators. We adopt this stabilizability condition as a standing assumption throughout the paper.
	Using the overtaking criterion, \cite{AL85} established the existence, uniqueness, and a feedback law for infinite-horizon periodic-signal tracking problems, and related the asymptotic behavior of the optimal trajectories to the long-run performance of finite-horizon optimizers.
	Additional technical details on periodic LQ problems can be found in \cite{WX17}.
	
	This paper is organized as follows. 
	In Section~\ref{pd}, we formulate a periodic optimal LQ problem and establish the periodic dichotomy transformation (see Theorem~\ref{dudic}). 
	In Section~\ref{Apl}, we characterize explicitly the extremals of the periodic optimal LQ problem (see Theorem~\ref{pee}) and present a numerical simulation.
	In Section~\ref{Ap2}, we prove the turnpike property for LQ problems with periodic continuous coefficients and provide a numerical simulation.
	In Section~\ref{Ap3}, we prove an exponential estimate for Riccati equation and give a representation for some Cauchy problem by the dichotomy transformation.

	\section{Periodic Dichotomy Transformation}\label{pd}
	This section develops a periodic dichotomy transformation to decouple the optimality system. 
	After formulating the periodic linear quadratic problem and reviewing the periodic Riccati and Lyapunov differential equations, we use their solutions to construct an explicit transformation in Theorem~\ref{dudic}. 
	
	\subsection{Periodic Linear Quadratic Problem}
	
	Let $n,m,k$ be positive integers and $\theta>0$. 
	We consider the following periodic optimal control problem:
	\begin{multline*}
		(\text{Per})_\theta:\quad
		\inf_{\vphantom{d}}\; \frac{1}{2}\int_0^\theta \Bigl( \bigl\|C(t)(y(t)-y_d(t))\bigr\|^2 \\
		+ \bigl\|Q^{1/2}(t)(u(t)-u_d(t))\bigr\|^2 \Bigr)\,\mathrm{d}t
	\end{multline*}
	over all pairs  $(y(\cdot),u(\cdot))\in C([0,\theta];\mathbb R^n)\times L^2(0,\theta;\mathbb R^m)$ satisfying the state equation
	\begin{equation*}
		\left\{
		\begin{split}
			&\dot{y}(t)=A(t)y(t)+B(t)u(t),\;\;t\in(0,\theta),\\
			&y(0)=y(\theta).\\
		\end{split}\right.
	\end{equation*}
	
	Here $A(\cdot) \in C(\mathbb R;\mathbb R^{n\times n})$, $B(\cdot)\in C(\mathbb R;\mathbb R^{n\times m})$, $C(\cdot)\in C(\mathbb R; \mathbb R^{k\times n})$ and $Q(\cdot)\in C(\mathbb R; \mathbb R^{m\times m})$ are time-periodic matrix-valued functions with period $\theta$. The matrix $Q(t)$ is assumed to be symmetric and positive definite for any $t\in\mathbb R$.
	In addition, the tracking terms $y_d(\cdot)\in C([0,\theta];\mathbb R^n)$ and $u_d(\cdot)\in L^2(0,\theta;\mathbb R^m)$ are assumed to be $\theta$-periodic as well.
	
	Existence and uniqueness results for such periodic optimal control problems $(\text{Per})_\theta$, along with necessary and sufficient optimality conditions, are well-established in the literature (see, for instance, \cite{RM}  and references therein).
	It is well known that a pair $(y_\theta(\cdot),u_\theta(\cdot))$ is optimal for $(\text{Per})_\theta$ if and only if there exists an adjoint state $\lambda_\theta(\cdot)\in C([0,\theta];\mathbb R^n)$ such that
	\begin{equation}\label{xiaoextremalsystLQ}
		\left\{
		\begin{aligned}
			\dot{y}_\theta(t) &= A(t)y_\theta(t) + B(t)Q^{-1}(t)B^*(t)\lambda_\theta(t) 
			+ B(t)u_d(t), \\
			\dot{\lambda}_\theta(t) &= C^*(t)C(t)y_\theta(t) - A^*(t)\lambda_\theta(t) 
			- C^*(t)C(t)y_d(t)
		\end{aligned}
		\right.
	\end{equation}
	with the periodic boundary conditions
	\begin{equation}\label{du12061}
		y_\theta(0)=y_\theta(\theta)\;\;\;\text{and}\;\;\;
		\lambda_\theta(0)=\lambda_\theta(\theta).
	\end{equation}
	Moreover, the optimal periodic control is given by 
	\begin{equation}\label{xiao61718}
		u_\theta(t)=u_d(t)+Q^{-1}(t)B^*(t)\lambda_\theta(t),\;\;\text{a.e.}\;\;t\in[0,\theta].
	\end{equation}
	Throughout this paper,  $A^*$ denotes the transpose of a matrix $A$.
	\subsection{Periodic Riccati and Lyapunov differential equations}\label{prlde}
	
	We begin by recalling the concepts of periodic exponential stabilizability and detectability.
	\begin{definition}\label{du1206}
		The periodic pair $(A(\cdot),B(\cdot))$ is called exponentially $\theta$-periodic stabilizable  
		if there exists a $\theta$-periodic continuous matrix function
		$K(\cdot)\in C(\mathbb R;\mathbb R^{m\times n})$ such that the following closed-loop system
		is exponentially stable:
		\begin{equation}
			\dot y(t)=\big(A(t)+B(t)K(t)\big)y(t),\;\;t>0.
		\end{equation} 
		The periodic pair $(A(\cdot), C(\cdot))$ is called exponentially $\theta$-periodic detectable 
		if $(A^*(\cdot),C^*(\cdot))$ is exponentially $\theta$-periodic stabilizable. 
	\end{definition}
	
	The following lemmas regarding the solvability of periodic matrix Riccati and Lyapunov differential equations are well-known (see, e.g., \cite[Theorems 2 and 4]{BCG2} for proofs).
	More precisely, Lemma \ref{ric} provides the necessary and sufficient conditions for the existence of periodic solutions to the periodic matrix Riccati differential equation; while  
	Lemma \ref{lya} addresses the periodic matrix Lyapunov differential equation.
	
	\begin{lemma}\label{ric}
		The periodic matrix Riccati differential equation
		\begin{equation}\label{du12052}
			\begin{split}
				\dot{P}(t)&+A^*(t)P(t)+P(t)A(t)\\
				&-P(t)B(t)Q^{-1}(t)B^*(t)P(t)\\
				&+C^*(t)C(t)=0,\quad t\in\mathbb{R}.
			\end{split}
		\end{equation}
		admits a unique symmetric, positive semidefinite and $\theta$-periodic
		solution $P(\cdot)$, such that the associated closed-loop matrix  $A(\cdot)-B(\cdot)Q^{-1}(\cdot)B^*(\cdot)P(\cdot)$ is exponentially stable\footnote{Indeed, if $A(\cdot)$ is periodic and bounded, then $\dot y(t)=A(t)y(t)$ is exponentially stable if and only if for every fixed $s$, the norm of associated transition matrix $\|\Phi(t,s)\|$ approaches to zero as $t$ tends to infinity.},
		if and only if
		$(A(\cdot),B(\cdot))$ is exponentially $\theta$-periodic stabilizable
		and $(A(\cdot), C(\cdot))$ is exponentially $\theta$-periodic detectable.
	\end{lemma}
	
	\begin{remark}\label{du120505}
		Suppose that $(A(\cdot),B(\cdot))$ is exponentially $\theta$-periodic stabilizable
		and $(A(\cdot), C(\cdot))$ is exponentially $\theta$-periodic detectable.
		Let $P(\cdot)$ be the solution of \eqref{du12052} as guaranteed by Lemma \ref{ric}, and let $\Psi(\cdot,\cdot)$ denote the transition matrix (or evolution operator) associated with the feedback matrix $A(\cdot)-B(\cdot)Q^{-1}(\cdot)B^*(\cdot)P(\cdot)$. 
		Then, there exist positive constants $c$ and $\lambda$ such that the periodicity property
		\begin{equation}\label{du12051}
			\Psi(t+\theta,s+\theta)=\Psi(t,s) \;\;\;\text{for all}\;\;t>s
		\end{equation}
		and the exponential stability estimate
		\begin{equation}
			\|\Psi(t,s)\|\leq ce^{-\nu (t-s)}\;\;\;\text{for all}\;\;t>s
		\end{equation}
		hold true.
	\end{remark}
	
	\begin{lemma}\label{lya}
		Suppose that $(A(\cdot),B(\cdot))$ is exponentially $\theta$-periodic stabilizable
		and $(A(\cdot), C(\cdot))$ is exponentially $\theta$-periodic detectable. Let $P(\cdot)$ be the unique solution of \eqref{du12052}. Then, the periodic matrix Lyapunov differential equation
		\begin{multline*}
			-\dot{E}(t)+\big(A(t)-B(t)Q^{-1}(t)B^*(t)P(t)\big)E(t)\\
			+E(t)\big(A(t)-B(t)Q^{-1}(t)B^*(t)P(t)\big)^*\\
			-B(t)Q^{-1}(t)B^*(t)=0,\;\;t\in\mathbb R,
		\end{multline*}
		admits a unique symmetric, bounded, and $\theta$-periodic solution $E(\cdot)$, which is explicitly given by
		\begin{equation}\label{lya1}
			E(t)= -\int_{-\infty}^t\Psi(t,s)B(s)Q^{-1}(s)B^*(s)\Psi^*(t,s)\,\mathrm{d}s,
		\end{equation}
		where $\Psi(\cdot,\cdot)$ is the transition matrix associated with $A(\cdot)-B(\cdot)Q^{-1}(\cdot)B^*(\cdot)P(\cdot)$.
	\end{lemma}
	
	\subsection{Periodic Dichotomy Transformation}
	
	In this subsection, we construct a linear, periodic dichotomy transformation on $\mathbb R^{2n}$ using the solutions of the periodic Riccati and Lyapunov differential equations introduced in Lemmas \ref{ric} and \ref{lya}.
	
	\begin{theorem}\label{dudic}
		Assume that $(A(\cdot),B(\cdot))$ is exponentially $\theta$-periodic stabilizable and that $(A(\cdot),C(\cdot))$ is exponentially $\theta$-periodic detectable. Then, the linear coupled system
		\begin{equation*}
			\begin{pmatrix}
				\dot{y}(t)\\\dot{\lambda}(t)
			\end{pmatrix}=
			\begin{pmatrix}
				A(t) & B(t)Q^{-1}(t)B^*(t)\\C^*(t)C(t)&-A^*(t)
			\end{pmatrix}\begin{pmatrix}
				y(t)\\ \lambda (t)
			\end{pmatrix},\; t\in\mathbb R,
		\end{equation*}
		can be decoupled into the block-diagonal form
		\begin{equation*}
			\begin{split}
				\begin{pmatrix}\dot{p}(t)\\ \dot{q}(t)\end{pmatrix}
				=
				\begin{pmatrix}
					L(t) & 0 \\
					0 & -L^*(t)
				\end{pmatrix}
				&
				\begin{pmatrix}p(t)\\ q(t)\end{pmatrix},\quad t\in\mathbb{R},
			\end{split}
		\end{equation*}
		where
		\begin{equation}\label{eq:A-BQ-1B*P}
			L(t)=A(t)-B(t)Q^{-1}(t)B^*(t)P(t),\quad t\in\mathbb{R}
		\end{equation}
		by using the linear, bounded, and $\theta$-periodic transformation
		\begin{equation}\label{ducan1}
			\begin{pmatrix}
				p(t)\\ q (t)
			\end{pmatrix}=
			\begin{pmatrix}
				I+E(t)P(t)&E(t)\\P(t)&I
			\end{pmatrix}
			\begin{pmatrix}
				y(t)\\  \lambda (t)
			\end{pmatrix},\quad t\in\mathbb R,
		\end{equation}
		where $P(\cdot)$ and $E(\cdot)$ are the solutions of the periodic Riccati and Lyapunov differential equations given in Lemmas \ref{ric} and \ref{lya}, respectively.
	\end{theorem}
	
	\begin{proof}
		The proof is inspired by the time-invariant case in \cite[Lemma 1]{TZZ}. Let $I$ denote the identity matrix in $\mathbb R^n$. For each $t\in\mathbb R$, we introduce the following matrices:
		\begin{align*}
			\mathcal T_1(t)&:=\begin{pmatrix} I&0\\P(t)&I \end{pmatrix}, &
			\mathcal T_2(t)&:=\begin{pmatrix} I&0\\-P(t)&I \end{pmatrix}, \\
			\mathcal T_3(t)&:=\begin{pmatrix} I&E(t)\\0&I \end{pmatrix}, &
			\mathcal T_4(t)&:=\begin{pmatrix} I&-E(t)\\0&I \end{pmatrix}.
		\end{align*}
		A direct computation shows that for all $t\in\mathbb R$,
		\begin{equation}\label{n3191}
			\begin{split}
				\mathcal{T}_1(t)\mathcal{T}_2(t)
				=\mathcal{T}_2(t)\mathcal{T}_1(t)
				&=\mathcal{T}_3(t)\mathcal{T}_4(t)
				=\mathcal{T}_4(t)\mathcal{T}_3(t)\\
				&=\begin{pmatrix}I&0\\0&I\end{pmatrix}.
			\end{split}
		\end{equation}
		Let $M(\cdot)$ denote the coefficient matrix of the original coupled system:
		\begin{equation*}
			M(t):=\begin{pmatrix}
				A(t) & B(t)Q^{-1}(t)B^*(t)\\[1mm]
				C^*(t)C(t) & -A^*(t)
			\end{pmatrix},\quad t\in\mathbb{R},
		\end{equation*}
		and
		\begin{equation*}
			L(t):=A(t)-B(t)Q^{-1}(t)B^*(t)P(t),\quad t\in\mathbb{R}.
		\end{equation*}
		Using the periodic differential equations satisfied by $P(\cdot)$ and $E(\cdot)$ (see Lemmas \ref{ric} and \ref{lya}), a straightforward computation yields the following identities for $t\in \mathbb{R}$:
		\begin{gather*}
			\dot{\mathcal T}_1(t) \mathcal T_2(t)= \begin{pmatrix} 0&0\\\dot{P}(t)&0 \end{pmatrix}, \qquad
			\dot{\mathcal T}_3(t)\mathcal T_4(t)= \begin{pmatrix} 0&\dot{E}(t)\\0&0 \end{pmatrix}, \\
			\mathcal T_1(t) M(t)\mathcal T_2(t)=
			\begin{pmatrix}
				L(t)& B(t)Q^{-1}(t)B^*(t)\\
				-\dot P(t)& -L^*(t)
			\end{pmatrix}.
		\end{gather*}
		Defining an intermediate matrix $N(\cdot)$ as
		\begin{equation*}
			\begin{split}
				N(t)&:=
				\dot{\mathcal T_1}(t) \mathcal T_2(t)+\mathcal T_1(t) M(t)\mathcal T_2(t)\\
				&\;=\begin{pmatrix}
					L(t)& B(t)Q^{-1}(t)B^*(t)\\
					0&-L^*(t)
				\end{pmatrix}, \quad t\in \mathbb{R},
			\end{split}
		\end{equation*}
		we further compute
		\begin{equation*}
			\mathcal T_3(t) N(t) \mathcal T_4(t)=\begin{pmatrix}
				L(t)& -\dot E(t) \\
				0&-L^*(t)
			\end{pmatrix}, \quad t\in \mathbb{R}.
		\end{equation*}
		Consequently, we obtain
		\begin{equation*}
			\dot{\mathcal T_3}(t) \mathcal T_4(t)+\mathcal T_3(t) N(t) \mathcal T_4(t) =
			\begin{pmatrix}
				L(t)& 0\\
				0&-L^*(t)
			\end{pmatrix}, \quad t\in \mathbb{R}.
		\end{equation*}
		We now define the transformation matrix $\mathcal T(\cdot)$ from $\mathbb R^{2n}$ to $\mathbb R^{2n}$ as
		\begin{equation}\label{dic}
			\mathcal T(t):=\mathcal T_3(t) \mathcal T_1(t)
			=\begin{pmatrix}
				I+E(t)P(t)&E(t)\\P(t)&I
			\end{pmatrix},\quad t\in\mathbb R.
		\end{equation}
		Clearly, $\mathcal{T}(\cdot)$ is $\theta$-periodic. From \eqref{n3191}, its inverse is given by
		\begin{equation*}
			\mathcal T(t)^{-1}=\mathcal T_2(t) \mathcal T_4(t)=
			\begin{pmatrix}
				I&-E(t)\\-P(t)&I+P(t)E(t)
			\end{pmatrix},\quad t\in\mathbb R.
		\end{equation*}
		Combining the identities derived above, we conclude that
		\begin{equation}\label{zuih1}
			\begin{split}
				&\dot{\mathcal T}(t) \mathcal T(t)^{-1}+\mathcal T(t) M(t) \mathcal T(t)^{-1}\\
				=&\dot{\mathcal T_3}(t) \mathcal T_1(t) \mathcal T_2(t) \mathcal T_4(t)
				+\mathcal T_3(t) \dot{\mathcal T_1}(t) \mathcal T_2(t) \mathcal T_4(t)\\
				&+\mathcal T_3(t) \mathcal T_1(t) M(t) \mathcal T_2(t) \mathcal T_4(t)\\
				=& \dot{\mathcal T_3}(t) \mathcal T_4(t)+\mathcal T_3(t) N(t) \mathcal T_4(t)\\
				=&
				\begin{pmatrix}
					L(t)& 0\\
					0&-L^*(t)
				\end{pmatrix},\quad t\in \mathbb{R}.
			\end{split}
		\end{equation}
		Hence, differentiating \eqref{ducan1} with respect to $t$ yields that
		\begin{equation*}
			\begin{split}
				\begin{pmatrix}
					\dot{p}(t)\\ \dot{q} (t)
				\end{pmatrix}
				&=\dot{\mathcal T}(t)\begin{pmatrix}
					y(t)\\ \lambda (t)
				\end{pmatrix}
				+\mathcal T(t) M(t)\begin{pmatrix}
					y(t)\\ \lambda (t)
				\end{pmatrix}\\
				&=\Big(\dot{\mathcal T}(t) \mathcal T(t)^{-1}+\mathcal T(t) M(t) \mathcal T(t)^{-1}\Big)
				\begin{pmatrix}
					p(t)\\ q (t)
				\end{pmatrix}, t\in\mathbb{R}.
			\end{split}
		\end{equation*}
		This, together with \eqref{zuih1}, completes the proof.
	\end{proof}
	\begin{remark}
		Theorem~\ref{dudic} provides a periodic dichotomy transformation which splits the $2n$-dimensional Hamiltonian system into two invariant $n$-dimensional subsystems. The subsystem associated with $L(\cdot)$ is exponentially stable in forward time (see Remark \ref{du120505}), while the one associated with $-L^*(\cdot)$ is exponentially stable in backward time. This result is a natural periodic extension of the time-invariant dichotomy transformation presented in \cite[Lemma 1]{TZZ}.
	\end{remark}
	
	\section{Representation of the Periodic Extremal}\label{Apl}
	
	\subsection{Characterization}
	
	The aim of this subsection is to provide an explicit characterization of the optimal solution for $(\text{Per})_\theta$. This is achieved by utilizing the periodic dichotomy transformation constructed from the solutions of the periodic Riccati and Lyapunov differential equations, under appropriate exponential stabilizability and detectability assumptions. The main result of this section is stated as follows.
	
	\begin{theorem}\label{pee}
		Assume that $(A(\cdot),B(\cdot))$ is exponentially $\theta$-periodic stabilizable and that $(A(\cdot),C(\cdot))$ is exponentially $\theta$-periodic detectable. Then, $(\text{Per})_\theta$ admits a unique periodic extremal $(y_\theta(\cdot),u_\theta(\cdot),\lambda_\theta(\cdot))$, which is explicitly given by 
		\begin{equation}\label{zq}
			\left\{
			\begin{aligned}
				y_\theta(t) &= z(t)-E(t)q(t), \\
				\lambda_\theta(t) &= -P(t)z(t)+\big(I+P(t)E(t)\big)q(t), \\
				u_\theta(t) &= u_d(t)+Q^{-1}(t)B^*(t)\lambda_\theta(t), 
			\end{aligned}
			\right.\quad t\in[0,\theta],
		\end{equation}
		where $I$ denotes the identity matrix on $\mathbb R^n$, and $P(\cdot)$ and $E(\cdot)$ are the solutions of the periodic matrix Riccati and Lyapunov differential equations given in Lemmas \ref{ric} and \ref{lya}, respectively. The auxiliary variables $z(\cdot)$ and $q(\cdot)$ are defined by 
		\begin{multline}\label{tian6}
			z(t) = \Psi(t,0)\big(I-\Psi(\theta,0)\big)^{-1}\int_0^\theta \Psi(\theta,\tau)g_1(\tau)\,\mathrm{d}\tau \\
			+ \int_0^t\Psi(t,\tau)g_1(\tau)\,\mathrm{d}\tau,\quad t \in[0,\theta],
		\end{multline}
		\begin{multline}\label{tian7}
			q(t) = -\Psi^*(\theta,t)\big(I-\Psi^*(\theta,0)\big)^{-1}\int_0^\theta \Psi^*(\tau,0)g_2(\tau)\,\mathrm{d}\tau \\
			- \int_t^\theta\Psi^*(\tau,t)g_2(\tau)\,\mathrm{d}\tau,\quad t \in[0,\theta],
		\end{multline}
		where $\Psi(\cdot,\cdot)$ is the transition matrix associated with $L(\cdot)$ in \eqref{eq:A-BQ-1B*P}, and
		\begin{align*}
			g_1(\cdot) &= \big(I+ E(\cdot) P(\cdot)\big)B(\cdot)u_d(\cdot) - E(\cdot)C^*(\cdot)C(\cdot)y_d(\cdot), \\
			g_2(\cdot) &= P(\cdot)B(\cdot)u_d(\cdot) - C^*(\cdot)C(\cdot)y_d(\cdot).
		\end{align*}
	\end{theorem}
	\begin{proof}
		Applying the dichotomy transformation constructed in Theorem~\ref{dudic},
		\begin{equation}\label{jin4}
			\begin{pmatrix}z(t)\\ q(t)\end{pmatrix}
			=
			\begin{pmatrix}
				I+E(t)P(t) & E(t)\\[1mm]
				P(t) & I
			\end{pmatrix}
			\begin{pmatrix}y_\theta(t)\\ \lambda_\theta(t)\end{pmatrix},
			\quad t\in[0,\theta],
		\end{equation}
		we decouple the Hamiltonian system \eqref{xiaoextremalsystLQ} into two independent differential equations:
		\begin{equation}\label{tian1}
			\dot z(t) = L(t)z(t) + g_1(t),\quad t\in(0,\theta),
		\end{equation}
		and
		\begin{equation}\label{tian2}
			\dot q(t) = -L^*(t)q(t) + g_2(t),\quad t\in(0,\theta), 
		\end{equation}
		where $L(\cdot)$ is given by \eqref{eq:A-BQ-1B*P}.
		
		Next, we determine the initial state $z(0)$ and the final state $q(\theta)$ for equations \eqref{tian1} and \eqref{tian2}, respectively, by employing the periodic boundary conditions \eqref{du12061}. This allows us to explicitly integrate \eqref{tian1} and \eqref{tian2}.
		
		From \eqref{jin4} and the periodic conditions \eqref{du12061}, it is evident that $z(0)=z(\theta)$. Using the Duhamel formula, integrating \eqref{tian1} yields
		\begin{equation}\label{jin7}
			z(\theta)=\Psi(\theta,0)z(0)+\int_0^\theta \Psi(\theta,\tau)g_1(\tau)\,\mathrm{d}\tau.
		\end{equation}
		Due to the periodicity and exponential stability of $\Psi(\cdot,\cdot)$ (see Remark~\ref{du120505}), the matrix $I-\Psi(\theta,0)$ is invertible. Consequently, we obtain
		\begin{equation*}
			z(0)=\big(I-\Psi(\theta,0)\big)^{-1}\int_0^\theta \Psi(\theta,\tau)g_1(\tau)\,\mathrm{d}\tau,
		\end{equation*}
		which implies that
		\begin{multline*}
			z(t)=\Psi(t,0)\big(I-\Psi(\theta,0)\big)^{-1}\int_0^\theta \Psi(\theta,\tau)g_1(\tau)\,\mathrm{d}\tau\\
			+\int_0^t\Psi(t,\tau)g_1(\tau)\,\mathrm{d}\tau,\quad t \in[0,\theta].
		\end{multline*}
		
		On the other hand, applying an analogous argument to \eqref{tian2} in backward time leads to
		\begin{multline*}
			q(t)=-\Psi^*(\theta,t)\big(I-\Psi^*(\theta,0)\big)^{-1}\int_0^\theta \Psi^*(\tau,0)g_2(\tau)\,\mathrm{d}\tau\\
			-\int_t^\theta\Psi^*(\tau,t)g_2(\tau)\,\mathrm{d}\tau,\quad t \in[0,\theta].
		\end{multline*}   
		
		Finally, using the inverse of the periodic linear transformation \eqref{jin4}, we deduce that
		\begin{equation*}
			\begin{pmatrix}
				y_\theta(t) \\ \lambda_\theta(t)
			\end{pmatrix}
			=
			\begin{pmatrix}
				I & -E(t) \\ -P(t) & I+P(t) E(t)
			\end{pmatrix}
			\begin{pmatrix}
				z(t) \\ q(t)
			\end{pmatrix},\quad t\in[0,\theta].
		\end{equation*}
		Combining this with the relation $u_\theta(\cdot) = u_d(\cdot)+Q^{-1}(\cdot)B^*(\cdot)\lambda_\theta(\cdot)$ yields \eqref{zq} and completes the proof.
	\end{proof}
	
	\subsection{A Numerical Simulation}\label{ans}
	In this subsection, the periodic solution of a two-dimensional optimal control problem is numerically computed using the periodic dichotomy approach, and the consistency of this method is illustrated by the turnpike property.
	
	Setting $\theta=2\pi$, we define the coefficient matrices of the periodic optimal control problem as follows:
	\begin{equation*}
		\begin{split}
			A(t)=
			\begin{pmatrix}
				\sin t & \cos^2 t\\
				e^{-\sin t} & -1+\cos t
			\end{pmatrix},\\
			B(t)=C(t)=Q(t)=
			\begin{pmatrix}
				1 & 0\\
				0 & 1
			\end{pmatrix},
		\end{split}
	\end{equation*}
	for $t\in[0,\theta]$, with the tracking terms given by
	\begin{equation*}
		y_d(t)= \begin{pmatrix} \sin t\\ \cos t \end{pmatrix},\quad
		u_d(t)= \begin{pmatrix} 0\\ 0 \end{pmatrix},\quad \text{for } t\in[0,\theta].
	\end{equation*}
	It is straightforward to verify that the pair $(A(\cdot),B(\cdot))$ is exponentially $\theta$-periodic stabilizable by setting the feedback matrix $K(\cdot)=-A(\cdot)-I$. A similar conclusion holds for the pair $(A^*(\cdot),C^*(\cdot))$ by setting $K(\cdot)=-A^*(\cdot)-I$.
	
	The initial step in the numerical implementation is to solve the periodic Riccati differential equation (PRDE). There are a number of approaches to the resolution of PRDEs, including the periodic generator method and the multiple shooting method. However, an alternative method is proposed, which, while less sophisticated, is more straightforward to implement. 
	
	\begin{figure}[htbp]
		\centering
		\begin{subfigure}[b]{0.48\linewidth}
			\includegraphics[width=\linewidth]{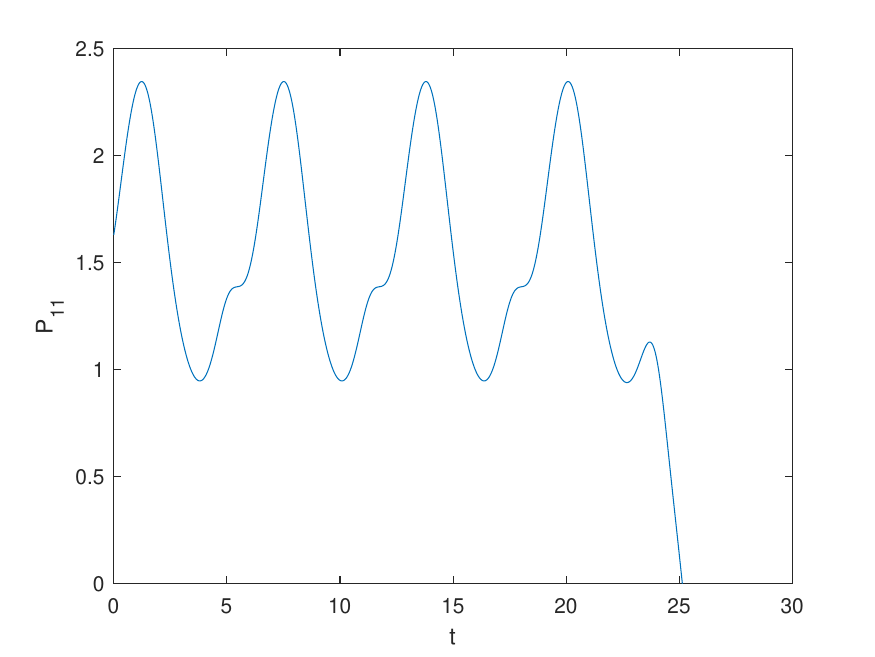}
			\caption{$P_{11}$ (backward)}
			\label{p11b}
		\end{subfigure}\hfill   % ← 关键空隙
		\begin{subfigure}[b]{0.48\linewidth}
			\includegraphics[width=\linewidth]{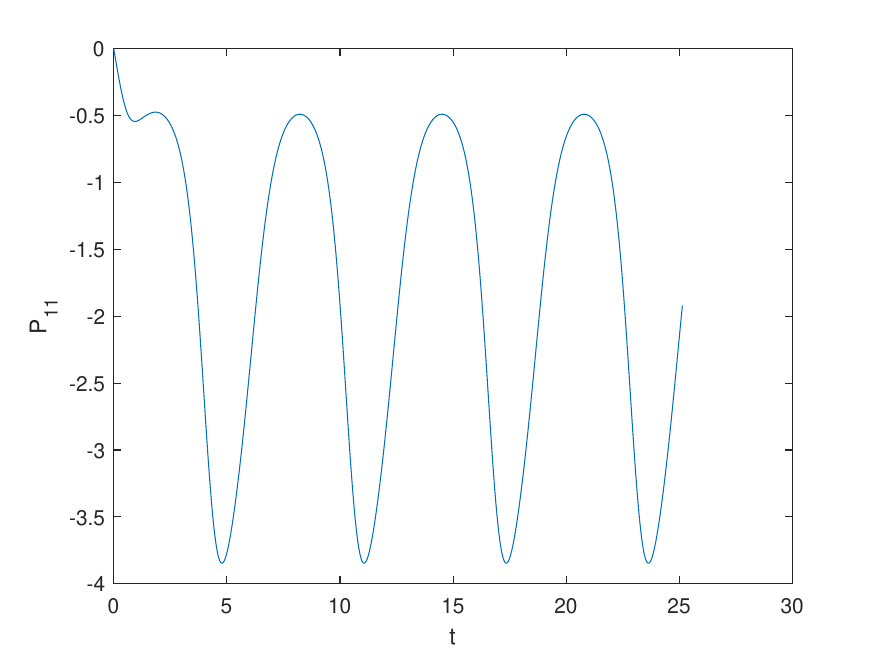}
			\caption{$P_{11}$ (forward)}
			\label{p11f}
		\end{subfigure}
		\caption{Two distinct periodic orbits of $P_{11}$}\label{p11}
	\end{figure}
	Let $P_{11}(\cdot)$ denote the $(1,1)$-entry of the solution to \eqref{du12052}. We observe empirically that when \eqref{du12052} is integrated backward from the final time with an arbitrary initial terminal state, $P_{11}(\cdot)$ converges asymptotically to a stable periodic orbit. This is illustrated in Figure \ref{p11b}, obtained using the MATLAB solver \texttt{ode45}. Conversely, when integrated forward in time, the solution converges to a completely different periodic orbit (Figure \ref{p11f}). The backward-convergent orbit corresponds to the positive semidefinite solution required by our theory, which is confirmed by the non-negativity of its computed characteristic multipliers (eigenvalues), as shown in Figure \ref{eig}.
	\begin{figure}[htbp]
		\centering
		\begin{subfigure}[b]{0.48\linewidth}
			\includegraphics[width=\linewidth]{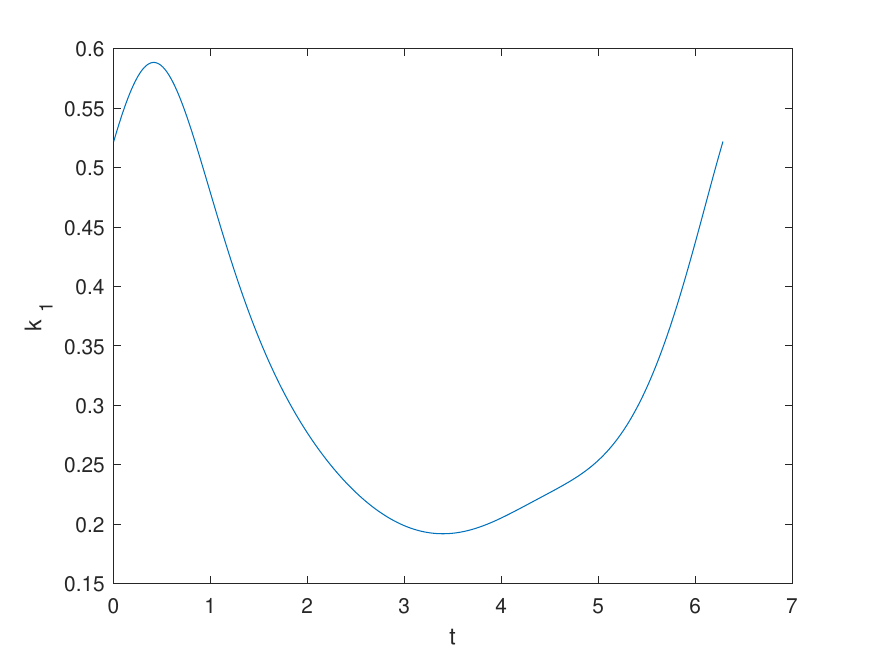}
			\caption{$k_1$}
			
		\end{subfigure}\hfill
		\begin{subfigure}[b]{0.48\linewidth}
			\includegraphics[width=\linewidth]{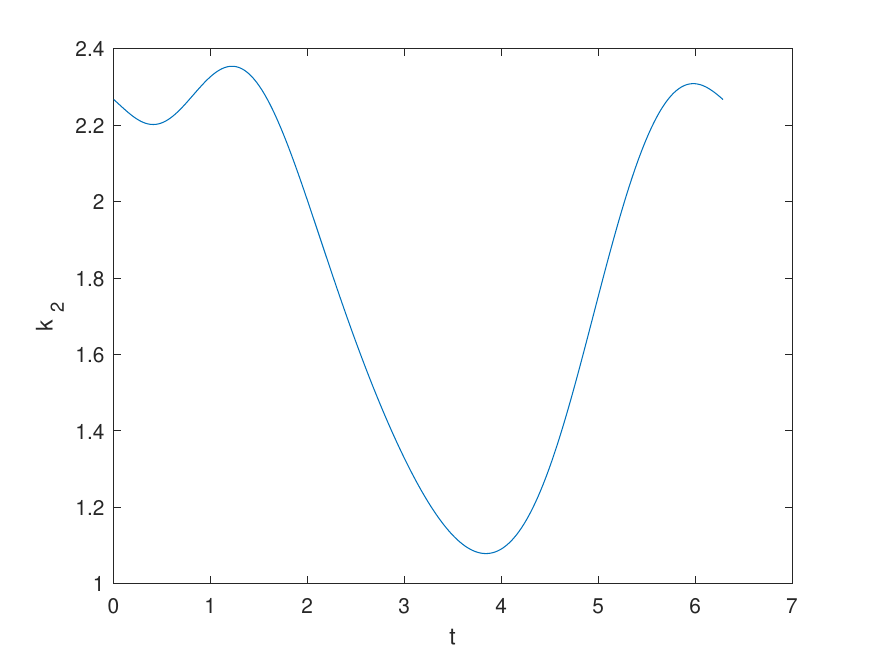}
			\caption{$k_2$}
			
		\end{subfigure}
		\caption{Characteristic multipliers (eigenvalues) of $P$}
		\label{eig}
	\end{figure}
	
	According to Lemma \ref{ric}, equation \eqref{du12052} admits a unique positive semidefinite periodic solution. Therefore, the periodic orbit isolated via backward integration is indeed the desired stabilizing solution. Using this methodology, we obtain the full solution $P(\cdot)$ to the PRDE, depicted in Figure \ref{figp}.
	
	\begin{figure}[htbp]
		\centering
		\begin{minipage}[b]{0.48\linewidth}
			\centering
			\includegraphics[width=\linewidth]{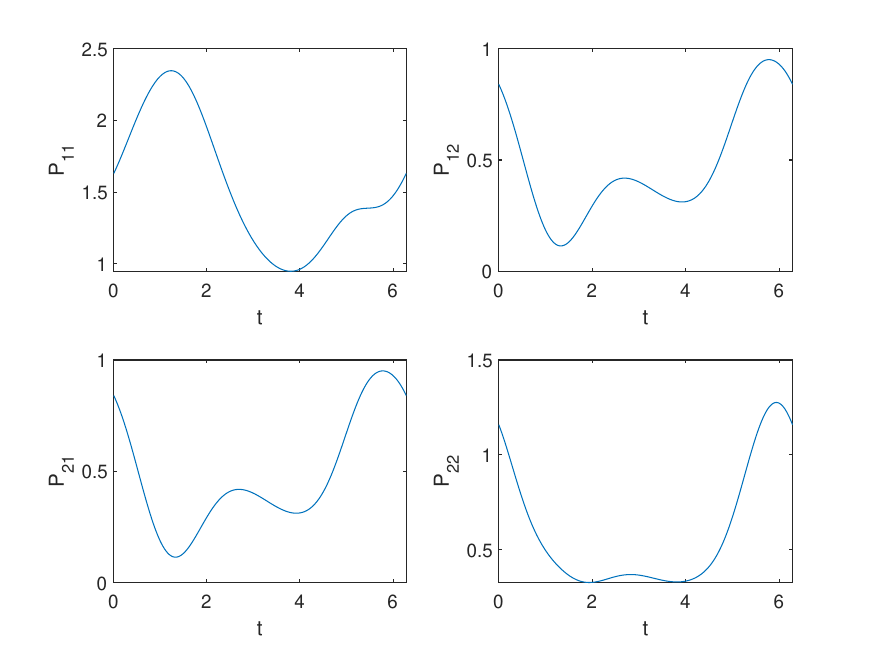}
			\caption{Matrix $P(\cdot)$}
			\label{figp}
		\end{minipage}
		\hfill
		\begin{minipage}[b]{0.48\linewidth}
			\centering
			\includegraphics[width=\linewidth]{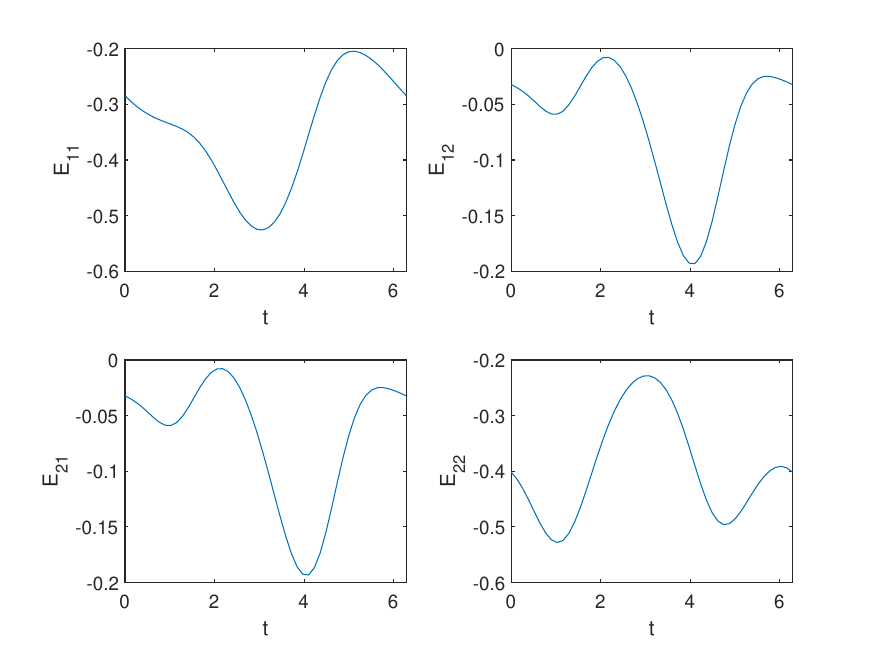}
			\caption{Matrix $E(\cdot)$}
			\label{fige}
		\end{minipage}
	\end{figure}
	
	The second step is to solve the periodic Lyapunov differential equation (PLDE). Although the PLDE is a special case of the PRDE and can be simulated similarly, its analytical solution is explicitly formulated in \eqref{lya1}. Thus, we compute $E(\cdot)$ directly using the MATLAB function \texttt{integral} (see Figure \ref{fige}).
	
	Finally, we compute the auxiliary variables $z(\cdot)$ and $q(\cdot)$ via \eqref{tian6} and \eqref{tian7}, and subsequently recover the optimal extremal through the transformation \eqref{zq}. The unique optimal extremal triplet $(y_\theta(\cdot),\lambda_\theta(\cdot),u_\theta(\cdot))$ is displayed in Figure \ref{yper}.
	
	\begin{figure*}[t]          % 跨双栏，只能 [t] 或 [b]
		\centering
		\begin{subfigure}[b]{0.32\textwidth}
			\includegraphics[width=\linewidth]{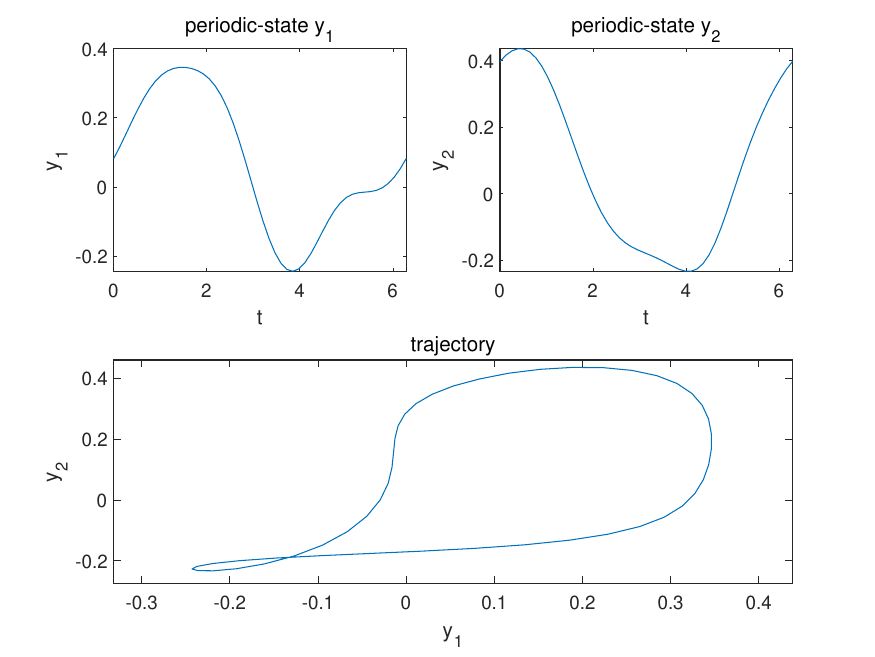}
			\caption{$y_\theta$}
		\end{subfigure}\hfill
		\begin{subfigure}[b]{0.32\textwidth}
			\includegraphics[width=\linewidth]{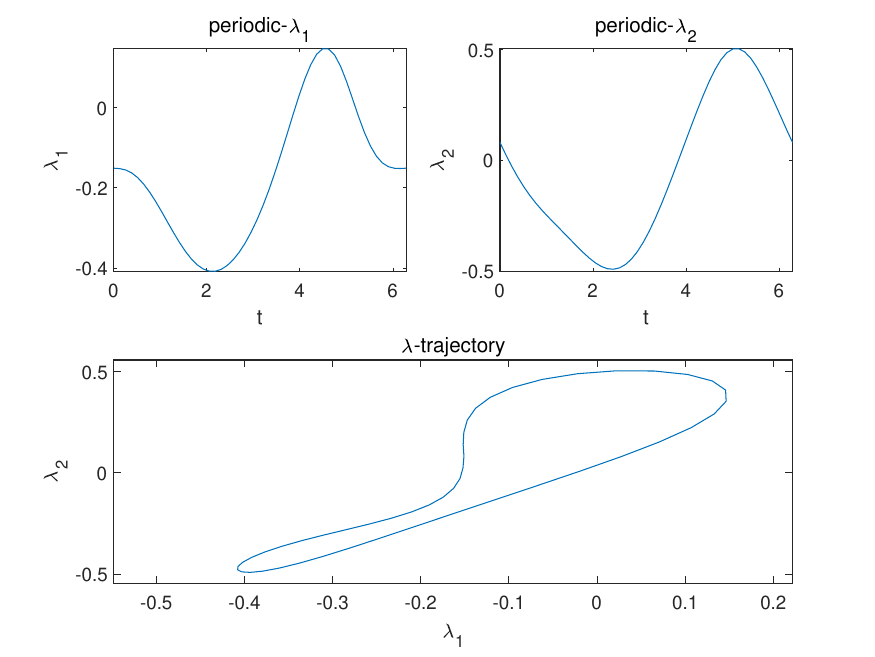}
			\caption{$\lambda_\theta$}
		\end{subfigure}\hfill
		\begin{subfigure}[b]{0.32\textwidth}
			\includegraphics[width=\linewidth]{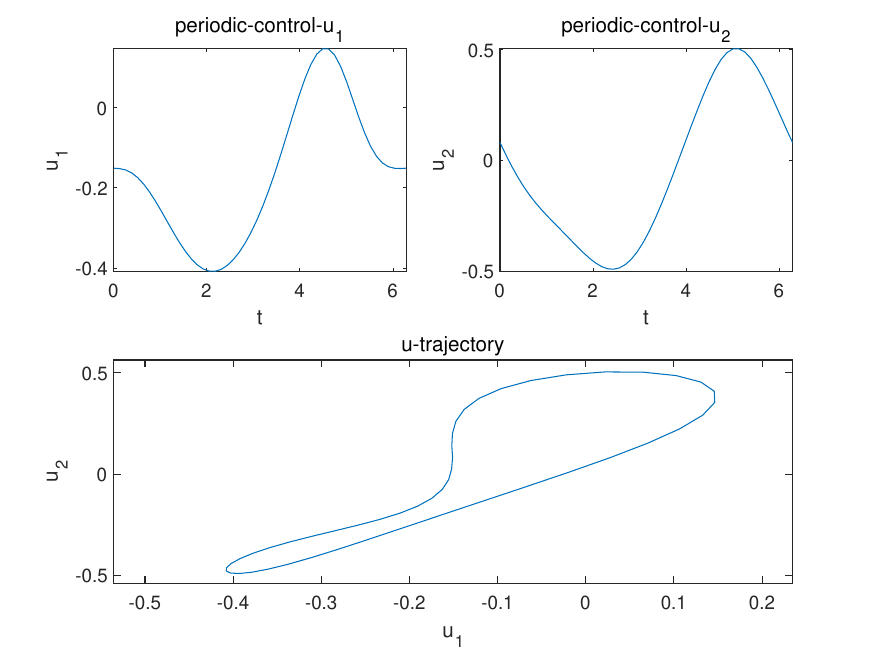}
			\caption{$u_\theta$}
		\end{subfigure}
		\caption{Periodic optimal extremal $(y_\theta, \lambda_\theta, u_\theta)$}
		\label{yper}
	\end{figure*}
	\begin{remark}
		Let us illustrate the repelling/attracting nature of the orbits shown in Figures \ref{p11b} and \ref{p11f} with a simple scalar example. Consider a one-dimensional Riccati equation with coefficients $A=B=Q=1$ and $C=\sqrt{3}$. The dynamics are governed by
		\begin{equation}\label{1-dp}
			\dot{p}(t) = -2p(t)+p^2(t)-3 = (p(t)-1)^2-4,\quad\text{for }t\in[0,\theta].
		\end{equation}
		It is clear from \eqref{1-dp} that $p_1 = -1$ and $p_2 = 3$ are the two equilibrium points. Analyzing the sign of the derivative, we observe that
		\begin{equation*}
			\left\{
			\begin{aligned}
				&\dot{p}(t) > 0, \quad \text{for } p(t) > p_2, \\
				&\dot{p}(t) < 0, \quad \text{for } p_1 < p(t) < p_2, \\
				&\dot{p}(t) > 0, \quad \text{for } p(t) < p_1,
			\end{aligned}
			\right.
		\end{equation*}
		for $ t\in[0,\theta]$.
		This shows that $p_1=-1$ is an attracting equilibrium (sink) while $p_2=3$ is a repelling equilibrium (source). Consequently, backward integration naturally isolates the positive definite solution $p_2=3$.
	\end{remark}
	\begin{remark}
		In the preceding discussion, we observed empirically the existence of both positive semidefinite and negative semidefinite periodic solutions. It is conceivable to derive a representation analogous to Theorem \ref{pee} by employing the dichotomy transformation introduced in \cite{WK}:
		\begin{equation*}
			\begin{pmatrix}
				p(\cdot)\\ q (\cdot)
			\end{pmatrix}=
			\begin{pmatrix}
				I & I \\
				P(\cdot) & N(\cdot)
			\end{pmatrix}
			\begin{pmatrix}
				y(\cdot)\\ \lambda (\cdot)
			\end{pmatrix},
		\end{equation*}
		where $P(\cdot)$ and $N(\cdot)$ denote the positive semidefinite and negative semidefinite solutions of the PRDE, respectively. However, the existence of the negative semidefinite solution in the general periodic case has not yet been fully addressed. This gap remains an interesting topic for future investigation.
	\end{remark}
	
	\section{Periodic Turnpike}\label{Ap2}
	
	\subsection{Periodic turnpike property}\label{pt}
	In this section, we show that the periodic dichotomy transformation developed in the previous sections provides a structural framework for establishing the periodic turnpike property for linear--quadratic (LQ) optimal control problems with periodic tracking trajectories over large time horizons.
	The exponential periodic turnpike property established in \cite{TZZ2} reveals that the optimal trajectory of the given LQ optimal control problem consists of three pieces: the first and last being transient short-time arcs, while the middle piece stays exponentially close to a periodic orbit associated with a periodic LQ optimal control problem. This property holds identically for the optimal control and the adjoint state.
	
	We consider the following optimal control problem, denoted by $(\text{LQ})^T$. For any $T>0$, consider the linear control system
	\begin{equation*}
		\left\{
		\begin{aligned}
			&\dot{y}(t)=A(t)y(t)+B(t)u(t),\quad t\in (0,T),\\
			&y(0)=y_0,
		\end{aligned}
		\right.
	\end{equation*}
	and the optimal control problem of minimizing the cost functional
	\begin{equation*}
		\begin{aligned}[b]
			(\text{LQ})^T:\quad
			\inf_{u(\cdot)}\; C_T(u) &= \frac{1}{2}\int_0^T \Bigl( \bigl\|C(t)(y(t)-y_d(t))\bigr\|^2 \\
			&\quad + \bigl\|Q^{1/2}(t)(u(t)-u_d(t))\bigr\|^2 \Bigr)\,\mathrm{d}t.
		\end{aligned}
	\end{equation*}
	The assumptions on the coefficient matrices are identical to those in problem $(\text{Per})_\theta$. It is well known that problem $(\text{LQ})^T$ admits a unique optimal solution pair $(y^T(\cdot),u^T(\cdot))$. Moreover, according to the Pontryagin maximum principle (see, e.g., \cite{L}), there exists an adjoint state $\lambda^T(\cdot) \in C([0,T];\mathbb R^n)$ such that
	\begin{equation}\label{eq:LQT_system}
		\left\{
		\begin{aligned}
			\dot{y}^T(t) &= A(t)y^T(t) + B(t)Q^{-1}(t)B^*(t)\lambda^T(t) + B(t)u_d(t), \\[1mm]
			\dot{\lambda}^T(t) &= C^*(t)C(t)y^T(t) - A^*(t)\lambda^T(t) - C^*(t)C(t)y_d(t),
		\end{aligned}
		\right.
	\end{equation}
	for $t\in (0,T)$, subject to the two-point boundary conditions
	\begin{equation}
		y^T(0)=y_0 \quad \text{and} \quad \lambda^T(T)=0.
	\end{equation}
	Moreover, the optimal control is given by 
	\begin{equation}
		u^T(t)=u_d(t)+Q^{-1}(t)B^*(t)\lambda^T(t),\quad \text{a.e. } t\in(0,T).
	\end{equation}
	
	We next give a representation of the solution to the adjoint equation in the following lemma.
	\begin{lemma}\label{varphi}
		The solution of the adjoint equation
		\begin{equation}\label{varphieq}
			\left\{
			\begin{aligned}
				&\dot{\varphi}(t)=-A^*(t) \varphi(t),\quad t\in [0, T],\\
				&\varphi(T)=\varphi_T ,
			\end{aligned}
			\right.
		\end{equation}
		is given by
		\begin{equation*}
			\varphi(t)=U^*_A(T,t)\varphi_T, \quad \forall t \in[0,T],
		\end{equation*}
		where $U_A(\cdot,\cdot)$ is the transition matrix (evolution operator) associated with $A(\cdot)$.
	\end{lemma}
	\begin{proof}
		For any $y_0 \in \mathbb R^n$ and each $s \in [0,T]$, consider the forward equation 
		\begin{equation*}
			\left\{
			\begin{aligned}
				&\dot{y}(t)=A(t) y(t),\quad t\in [s, T],\\
				&y(s)=y_0.
			\end{aligned}
			\right.
		\end{equation*}
		By definition, its solution is given by $y(t)=U_A(t,s)y_0$.
		Differentiating the inner product $\langle y(t), \varphi(t)\rangle$ with respect to $t$ yields
		\begin{equation*}
			\frac{\mathrm{d}}{\mathrm{d}t}\left\langle y(t), \varphi(t)\right\rangle = \left\langle A(t)y(t), \varphi(t)\right\rangle + \left\langle y(t), -A^*(t)\varphi(t)\right\rangle = 0.
		\end{equation*}
		Integrating over $t \in [s,T]$, we obtain
		\begin{equation*}
			\left\langle y(T), \varphi(T)\right\rangle=\left\langle y(s), \varphi(s)\right\rangle.
		\end{equation*}
		Hence we have 
		\begin{equation*}
			\left\langle U_A(T,s)y_0, \varphi_T\right\rangle=\left\langle y_0, \varphi(s)\right\rangle,
		\end{equation*}
		which implies that
		\begin{equation*}
			\left\langle y_0, U^*_A(T,s)\varphi_T\right\rangle=\left\langle y_0, \varphi(s)\right\rangle,\quad \forall s \in [0,T].
		\end{equation*}
		Due to the arbitrariness of $y_0$, we conclude that $\varphi(s)=U^*_A(T,s)\varphi_T$ for all $s \in[0,T]$.
	\end{proof}
	
	Before presenting the proof of the exponential periodic turnpike property, we establish an essential stability estimate in Lemma~\ref{upperbound}. This lemma extends the time-invariant result of \cite[Lemma 2]{TZZ} to time-varying systems.
	\begin{lemma}\label{upperbound}
		Assume that $(A(\cdot),B(\cdot))$ is exponentially $\theta$-periodic stabilizable, and that $(A(\cdot),C(\cdot))$ is exponentially $\theta$-periodic detectable. Then, there exists a constant $c>0$, independent of $T$, such that the stability estimate 
		\begin{equation}\label{7271}
			\|y(T)\|+\|\lambda(0)\|\leq c\big (\|y(0)\|+\|\lambda(T)\|\big)
		\end{equation}
		holds for any solution $(y(\cdot),\lambda(\cdot))\in C([0,T];\mathbb R^n)\times C([0,T];\mathbb R^n)$ of the linear coupled system
		\begin{equation}\label{6181}
			\left\{
			\begin{aligned}
				\dot y(t) &= A(t) y(t)+B(t)Q^{-1}(t)B^*(t)\lambda(t),\\
				\dot\lambda(t)&= C^*(t)C(t)y(t)-A^*(t)\lambda(t),
			\end{aligned}
			\right.\quad t\in(0,T).
		\end{equation}
	\end{lemma}
	\begin{proof}
		Since $(A(\cdot),C(\cdot))$ is exponentially $\theta$-periodic detectable, 
		there exists a continuous $\theta$-periodic feedback matrix $K_1(\cdot)\in C(\mathbb R;\mathbb R^{n\times k})$ such that the closed-loop system
		\begin{equation}
			\dot z(t)=\big(A(t)+K_1(t)C(t)\big)z(t),\quad t>0,
		\end{equation}is exponentially stable.
		Let $\varphi(\cdot) \in C\big([0,T]; \mathbb{R}^n\big)$ be the unique solution of 
		\begin{equation*}
			\left\{
			\begin{aligned}
				&\dot{\varphi}(t)=-\big(A^*(t)+C^*(t) K^*_1(t)\big) \varphi(t), \quad t \in[0, T], \\
				&\varphi(T)=y(T).
			\end{aligned}
			\right.
		\end{equation*}
		By Lemma~\ref{varphi} and the exponential decay of the associated transition matrix, 
		there exists a constant $C_1 > 0$ (independent of $T$) such that
		\begin{equation}
			\|\varphi (t)\| \leq C_1 \|y(T)\|,\quad \forall t\in [0,T].
		\end{equation}
		Differentiating the inner product $\langle \varphi(t), y(t)\rangle$, we have
		\begin{equation*}
			\begin{split}
				&\frac{\mathrm{d}}{\mathrm{d}t}\langle \varphi(t), y(t)\rangle \\&= \langle-(A^*(t)+C^*(t) K^*_1(t)) \varphi(t), y(t)\rangle \\
				&\quad + \langle \varphi(t), A(t) y(t)+B(t)Q^{-1}(t)B^*(t)\lambda(t)\rangle\\
				&= \langle \varphi(t), -K_1(t)C(t) y(t)+B(t)Q^{-1}(t)B^*(t)\lambda(t)\rangle.
			\end{split}
		\end{equation*}
		Integrating this over $t \in [0,T]$ yields
		\begin{equation*}
			\begin{split}
				&\left\|y(T) \right\| ^2=\langle \varphi(0), y(0)\rangle\\
				&+\int_0^T \langle \varphi(t), -K_1(t)C(t) y(t)
				+B(t)Q^{-1}(t)B^*(t)\lambda(t)\rangle  \,\mathrm{d}t.
			\end{split}
		\end{equation*} 
		We deduce that
		\begin{equation*}
			\begin{split}
				&\int_0^T \bigl| \langle \varphi(t), K_1(t)C(t) y(t)\rangle \bigr|\,\mathrm{d}t \\&\leq \left( \int_0^T \| K_1 C y\|^2\,\mathrm{d}t \right)^{1/2} \left( \int_0^T \| \varphi\|^2\,\mathrm{d}t \right)^{1/2} \\
				&\leq C_2 \max_{t \in [0,\theta]} \|K_1(t)\| \|y(T)\| \left( \int_0^T \|C y\|^2\,\mathrm{d}t \right)^{1/2}, 
			\end{split}
		\end{equation*}
		and similarly,
		\begin{equation*}
			\begin{split}
				&\int_0^T \bigl| \langle \varphi(t), B(t)Q^{-1}(t)B^*(t)\lambda(t)\rangle \bigr|\,\mathrm{d}t\\ &\leq C_3 \max_{t \in [0,\theta]} \|B(t)Q^{-1/2}(t)\| \|y(T)\| \\
				&\quad \times \left( \int_0^T \|Q^{-1/2}B^*\lambda\|^2\,\mathrm{d}t \right)^{1/2},
			\end{split}
		\end{equation*}
		for  some positive constant $C_2$ and $C_3$  (independent of $T$). It follows that
		\begin{equation}\label{yT}
			\begin{split}
				\left\|y(T) \right\|^2 
				&\leqslant C_4\Bigg( \left\| y(0) \right\|^2 +\int_0^T \left\| C(t) y(t)\right\|^2 \,\mathrm{d}t\\
				&\quad + \int_0^T\left\|Q^{-1/2}(t)B^*(t)\lambda(t) \right\|^2 \,\mathrm{d}t\Bigg),
			\end{split}
		\end{equation}
		for  some positive constant $C_4$ (independent of $T$).
		
		Similarly, since $(A(\cdot),B(\cdot))$ is exponentially $\theta$-periodic stabilizable, we obtain from the second equation in $\eqref{6181}$ that
		\begin{equation}\label{lambda0}
			\begin{split}
				\left\|\lambda(0) \right\|^2 
				&\leqslant C_5\Bigg( \left\| \lambda(T) \right\|^2 +\int_0^T \left\| C(t) y(t)\right\|^2 \,\mathrm{d}t\\
				&\quad + \int_0^T\left\|Q^{-1/2}(t)B^*(t)\lambda(t) \right\|^2 \,\mathrm{d}t\Bigg),
			\end{split} 
		\end{equation}
		for  some positive constant $C_5$ (independent of $T$).
		
		Let $C_6=\max\{C_4,C_5\}$. 
		Differentiating the duality pairing $\langle y(t), \lambda(t)\rangle$ along the trajectories of $\eqref{6181}$, we obtain
		\begin{equation*}
			\begin{split}
				\frac{\mathrm{d}}{\mathrm{d}t}\langle y(t), \lambda(t)\rangle &= \langle A(t) y(t)+B(t)Q^{-1}(t)B^*(t)\lambda(t), \lambda(t)\rangle \\
				&\quad + \langle y(t), C^*(t)C(t)y(t)-A^*(t)\lambda(t)\rangle\\
				&= \|Q^{-1/2}(t)B^*(t)\lambda(t)\|^2 + \|C(t) y(t)\|^2.
			\end{split}
		\end{equation*}
		Integrating over $t \in [0,T]$ and applying the Cauchy-Schwarz and Young's inequalities, we have
		\begin{equation*}
			\begin{split}
				&\int_0^T \|C(t) y(t)\|^2\,\mathrm{d}t + \int_0^T \|Q^{-1/2}(t)B^*(t)\lambda(t)\|^2\,\mathrm{d}t \\&= \langle y(T),\lambda(T)\rangle - \langle y(0),\lambda(0)\rangle \\
				&\leq \|y(T)\| \|\lambda(T)\| + \|y(0)\| \|\lambda(0)\|\\
				&\leq \frac{1}{4C_6}\|y(T)\|^2 + C_6\|\lambda(T)\|^2 + C_6\|y(0)\|^2 + \frac{1}{4C_6}\|\lambda(0)\|^2.
			\end{split}
		\end{equation*}
		This, along with $\eqref{yT}$ and $\eqref{lambda0}$, implies that
		\begin{equation*}
			\begin{split}
				&\quad \int_0^T \left\| C(t) y(t)\right\|^2 \,\mathrm{d}t + \int_0^T\left\|Q^{-1/2}(t)B^*(t)\lambda(t) \right\|^2 \,\mathrm{d}t\\
				&\leqslant C_7\left( \left\|y(0) \right\|^2+\left\| \lambda(T) \right\|^2\right),
			\end{split}
		\end{equation*}
		for some positive constant $C_7$ independent of $T$. Re-applying $\eqref{yT}$ and $\eqref{lambda0}$ leads directly to \eqref{7271}, which completes the proof.
	\end{proof}
	
	It should be emphasised that the turnpike property in the following theorem has been established in \cite[Theorem~2.1]{TZZ2}. Here we give an alternative proof relying on the periodic dichotomy transformation.
	\begin{theorem}\label{tp}
		Assume that $(A(\cdot),B(\cdot))$ is exponentially $\theta$-periodic stabilizable, and that $(A(\cdot),C(\cdot))$ is exponentially $\theta$-periodic detectable. Then, there exist positive constants $c$ and $\nu$ such that, for any $T>0$,
		\begin{multline}\label{19}
			\|y^T(t)-y_\theta(t)\| + \|\lambda^T(t)-\lambda_\theta(t)\| + \|u^T(t)-u_\theta(t)\| \\ \leq c(e^{-\nu t} + e^{-\nu (T-t)} ), \quad \forall t \in (0,T),
		\end{multline}
		where $(y_\theta(\cdot),\lambda_\theta(\cdot),u_\theta(\cdot))$ is the optimal extremal of the periodic problem $(\text{Per})_\theta$ extended by $\theta$-periodicity.
	\end{theorem}
	\begin{proof}
		Setting
		\begin{equation*}
			\begin{split}
				&\delta y(\cdot)=y^T(\cdot)-y_\theta(\cdot),\\
				&\delta \lambda(\cdot)=\lambda^T(\cdot)-\lambda_\theta(\cdot),\\
				&\delta u(\cdot)=u^T(\cdot)-u_\theta(\cdot).
			\end{split}
		\end{equation*}
		From the definition, the following holds
		\begin{equation*}
			\begin{pmatrix}
				\dot{\delta y}(t)\\\dot{\delta \lambda}(t)
			\end{pmatrix}=
			\begin{pmatrix}
				A(t) & B(t)Q^{-1}(t)B^*(t)\\C^*(t)C(t)&-A^*(t)
			\end{pmatrix}\begin{pmatrix}
				\delta y(t)\\ \delta \lambda (t)
			\end{pmatrix},
		\end{equation*}
		for every $t\in(0,T)$, with the terminal conditions
		\begin{equation*}
			\delta y(0) = y_0-y_\theta(0), \quad \text{and} \quad \delta \lambda(T) = -\lambda_\theta\left(T-\lfloor T/\theta \rfloor \theta \right),
		\end{equation*}
		where the $\left\lfloor \cdot \right\rfloor$ is the floor function. 
		Using the periodic dichotomy transformation \eqref{ducan1} introduced in Theorem \ref{dudic}, we decouple the system into
		\begin{equation*}
			\begin{pmatrix}
				\dot{p}(t)\\ \dot{q}(t)
			\end{pmatrix} =
			\begin{pmatrix}
				L(t) & 0 \\
				0 & -L^*(t)
			\end{pmatrix}
			\begin{pmatrix}
				p(t)\\ q(t)
			\end{pmatrix},\quad t\in(0,T),
		\end{equation*}
		where $L(\cdot) = A(\cdot)-B(\cdot)Q^{-1}(\cdot)B^*(\cdot)P(\cdot)$. We immediately infer that
		\begin{equation}
			p(t) = \Psi(t,0)p(0), \quad \text{and} \quad q(t) = -\Psi^*(T,t)q(T),
		\end{equation}
		for all $t\in(0,T)$, where $\Psi(\cdot,\cdot)$ is the transition matrix associated with $L(\cdot)$. Note that the initial and terminal states in the transformed coordinates are
		\begin{equation}
			\begin{aligned}
				p(0) &= \big(I+E(0)P(0)\big)\delta y(0) + E(0)\delta \lambda(0), \\
				q(T) &= P(T)\delta y(T) + \delta \lambda(T).
			\end{aligned}
		\end{equation}
		By Lemma~\ref{upperbound}, we have the stability estimate for the time-varying system:
		\begin{equation}\label{22}
			\|\delta y(T)\| + \|\delta \lambda(0)\| \leq c \big(\|\delta y(0)\| + \|\delta \lambda(T)\|\big),
		\end{equation} for some constant $c > 0$ independent of $T$. As stated in Remark \ref{du120505}, there exists a positive constant $\nu$ such that the transition matrix satisfies the exponential decay bounds
		\begin{equation}\label{23}
			\|\Psi(t,0)\| \leq ce^{-\nu t} \quad \text{and} \quad \|\Psi(T,t)\| \leq ce^{-\nu (T-t)}.
		\end{equation}
		Combining the inverse of the dichotomy transformation with \eqref{22} and \eqref{23} directly yields the estimate
		\begin{equation*}
			\|\delta y(t)\| + \|\delta \lambda(t)\| \leq c(e^{-\nu t} + e^{-\nu (T-t)}), \quad t\in(0,T).
		\end{equation*}
		Finally, the error estimate for the control follows from the relation $\delta u(\cdot) = Q^{-1}(\cdot)B^*(\cdot)\delta \lambda(\cdot)$, which concludes the proof of \eqref{19}.
	\end{proof}
	\begin{figure*}[t]        % t=top, b=bottom, 不能写 h
		\centering
		\begin{subfigure}[b]{0.32\linewidth}   % 用 \textwidth 做基准
			\includegraphics[width=\linewidth]{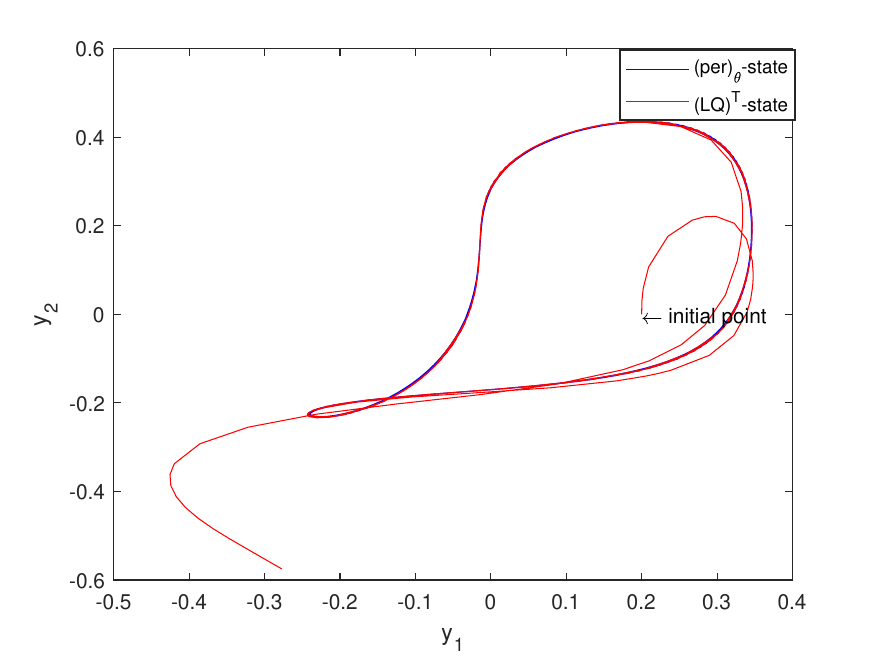}
			\caption{State comparison}
		\end{subfigure}\hfill
		\begin{subfigure}[b]{0.32\linewidth}
			\includegraphics[width=\linewidth]{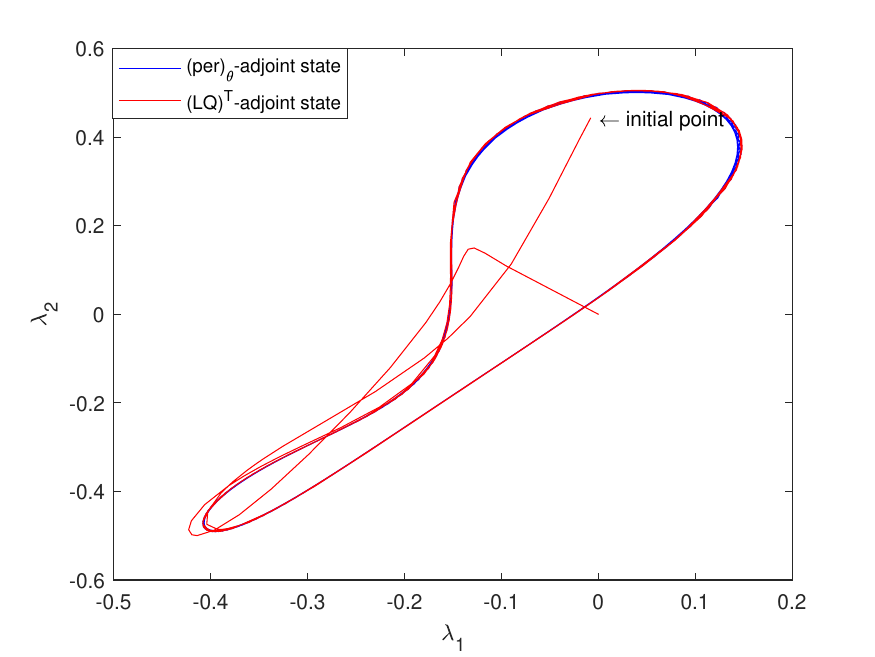}
			\caption{Adjoint state comparison}
		\end{subfigure}\hfill
		\begin{subfigure}[b]{0.32\linewidth}
			\includegraphics[width=\linewidth]{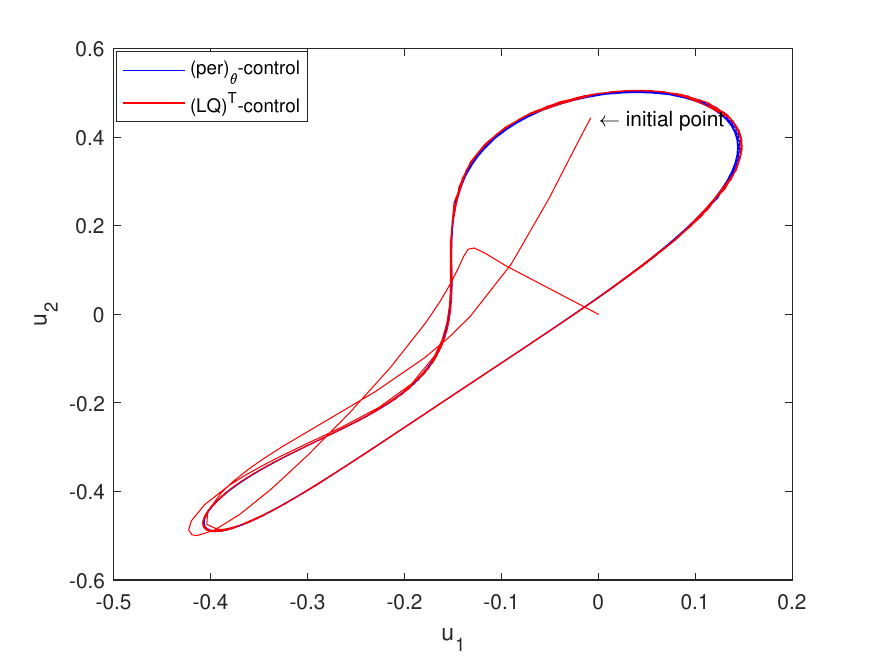}
			\caption{Control comparison}
		\end{subfigure}
		\caption{Turnpike phenomenon}
		\label{lqper}
	\end{figure*}
	\begin{remark}
		It is worth noting from the proof that the exponential decay rate $\nu$ in \eqref{19} is explicitly governed by the transition matrix associated with the optimal closed-loop matrix $A(\cdot)-B(\cdot)Q^{-1}(\cdot)B^*(\cdot)P(\cdot)$.
	\end{remark}
	As a consequence of Theorem \ref{tp}, the optimal extremal $(y_\theta(\cdot),\lambda_\theta(\cdot),u_\theta(\cdot))$ of the periodic problem can be utilized to approximate the optimal value of the finite-horizon problem $(\text{LQ})^T$ when $T$ is sufficiently large. Inspired by \cite[Theorem 2.2]{PZ13}, we establish the following asymptotic limit for the average cost.
	
	\begin{corollary}\label{prop:limit}
		Under the assumptions of Theorem \ref{tp}, the minimal average cost of $(\text{LQ})^T$ converges to the minimal average cost of $(\text{Per})_\theta$ as $T \to \infty$, i.e.,
		\begin{equation}\label{alim}
			\lim_{T \to \infty} \frac{C_{T}(u^T)}{T} = \frac{C_{\theta}(u_\theta)}{\theta}.
		\end{equation}
	\end{corollary}
	\begin{proof}
		Let $f^0(y,u) = \frac{1}{2} \|C(t)(y-y_d(t))\|^2 + \frac{1}{2} \|Q^{1/2}(t)(u-u_d(t))\|^2$ denote the running cost. According to Theorem \ref{tp}, for any given $\epsilon > 0$, there exists a constant $\tau > 0$ such that
		\begin{multline}\label{tpepsion}
			\|y^T(t)-y_\theta(t)\|+ \|\lambda^T(t)-\lambda_\theta(t)\| 
			+\|u^T(t)-u_\theta(t)\| \\ \leq \epsilon ,\quad t\in[\tau,T-\tau],
		\end{multline}
		while over the entire interval $t\in[0,T]$, the difference is uniformly bounded independent of $T$:
		\begin{multline}\label{tpM}
			\|y^T(t)-y_\theta(t)\|+ \|\lambda^T(t)-\lambda_\theta(t)\| 
			+\|u^T(t)-u_\theta(t)\|  \\
			\leq c(1+e^{-T})\leq C ,\quad t\in[0,T].
		\end{multline}
		Define $n_0 = \max \{ n \in \mathbb{N} \mid n\theta < T-2\tau \}$ and $T_0 = \tau + n_0\theta$. We decompose the difference between the average costs as follows:
		\begin{equation}
			\begin{split}
				&C_{T}(u^T) - \frac{T}{\theta} C_{\theta}(u_\theta) \\&= \int_\tau^{T_0} \big(f^0(y^T(t),u^T(t)) - f^0(y_\theta(t),u_\theta(t))\big)\,\mathrm{d}t \\
				&\quad + \left(\int_0^\tau + \int_{T_0}^T\right) f^0(y^T(t),u^T(t))\,\mathrm{d}t - \left(\frac{T}{\theta} - n_0\right)C_{\theta}(u_\theta) \\
				&= I_1 + I_2 + I_3.
			\end{split}
		\end{equation}
		Since the coefficients are continuous and periodic, and the variables are uniformly bounded due to \eqref{tpM}, the running cost $f^0$ is locally Lipschitz. Thus, from \eqref{tpepsion}, there exists a constant $M_3 > 0$ such that
		\begin{equation*}
			\big| f^0(y^T(t),u^T(t)) - f^0(y_\theta(t),u_\theta(t)) \big| \leq M_3 \epsilon, \quad t \in [\tau, T_0],
		\end{equation*}
		which implies $|I_1| \leq M_3 (T-2\tau)\epsilon \leq M_3 T \epsilon$. Furthermore, the uniform boundedness \eqref{tpM} guarantees that $f^0(y^T(t),u^T(t))$ is bounded by some constant $M_4$. The integration domains for $I_2$ have a total length of $\tau + (T - T_0) \leq \tau + (\tau + \theta) = 2\tau + \theta$. Hence,
		\begin{equation}\label{i2}
			|I_2| \leq (2\tau + \theta)M_4 := M_5.
		\end{equation}
		For the remaining term $I_3$, since $0 \leq \frac{T}{\theta} - n_0 \leq \frac{2\tau}{\theta} + 1$, we have
		\begin{equation}\label{i3}
			|I_3| \leq \left(1+\frac{2\tau}{\theta}\right) C_{\theta}(u_\theta) := M_6.
		\end{equation}
		Combining these estimates yields
		\begin{equation}
			\left| \frac{C_{T}(u^T)}{T} - \frac{C_{\theta}(u_\theta)}{\theta} \right| \leq \frac{|I_1| + |I_2| + |I_3|}{T} \leq M_3 \epsilon + \frac{M_5 + M_6}{T}.
		\end{equation}
		Taking the limit as $T \to \infty$, the right-hand side is bounded by $M_3 \epsilon$. Since $\epsilon > 0$ is arbitrary, the limit \eqref{alim} follows.
	\end{proof}

	\subsection{Numerical Simulation}
	
	We provide a numerical example to illustrate the preceding theoretical discussion. Consider the finite-horizon problem $(\text{LQ})^T$ governed by the following coefficients:
	\begin{equation*}
		\begin{split}
			A(t)=
			\begin{pmatrix}
				\sin t & \cos^2 t\\
				e^{-\sin t} & -1+\cos t
			\end{pmatrix},\\
			B(t)=C(t)=Q(t)=
			\begin{pmatrix}
				1 & 0\\
				0 & 1
			\end{pmatrix}.
		\end{split}
	\end{equation*} The tracking terms are given by
	\begin{equation*}
		y_d(t)=
		\begin{pmatrix}
			\sin t\\
			\cos t
		\end{pmatrix},\quad
		u_d(t)=
		\begin{pmatrix}
			0\\
			0
		\end{pmatrix}.
	\end{equation*}
	All coefficient matrices of this optimal control problem coincide identically with those defined in the periodic setting of Subsection \ref{ans}. The optimal extremal $(y^T(\cdot),\lambda^T(\cdot),u^T(\cdot))$ is computed via the MATLAB function \texttt{bvp4c}, setting the fixed time horizon to $T=30$ and the initial state to $y(0)=(0.2,0)^\top$, while leaving the final state free (which corresponds to $\lambda(T)=0$).
	
	As depicted in Figure \ref{lqper}, the optimal finite-horizon extremal $(y^T(\cdot),\lambda^T(\cdot),u^T(\cdot))$ (plotted in red) remains close to the periodic optimal extremal $(y_\theta(\cdot),\lambda_\theta(\cdot),u_\theta(\cdot))$ (plotted in blue) for most of the time interval, except near the initial time and the final time. This behavior is a characteristic feature of the exponential turnpike phenomenon, and it validates the periodic dichotomy as an effective theoretical and computational tool.
	
	Furthermore, if a rapid estimate of the minimal cost is required, one can leverage the asymptotic approximation formula derived from Proposition \ref{prop:limit}:
	
	It is convenient to use the following approximation formula derived from \eqref{alim}
	\begin{equation*}
		C_T(u^T) \approx  \frac{T}{\theta} C_\theta(u_\theta),
	\end{equation*}
	which holds accurately for sufficiently large $T$. In this specific simulation ($T=30, \theta=2\pi$), the actual cost $C_T(u^T)$ equals 21.4649, while the approximation $(T/\theta)C_\theta(u_\theta)$ evaluated via the MATLAB function \texttt{integral} yields 21.6937, demonstrating an excellent agreement.
	
	\section{Other Applications}\label{Ap3}
	\subsection{Asymptotic Behavior of the Riccati Equation}
	Consider the following matrix Riccati differential equation with a terminal condition:
	\begin{equation}\label{PG}
		\left\{
		\begin{aligned}
			&\dot{P}(t)+A^*(t)P(t)+P(t)A(t)-P(t)B(t)Q^{-1}(t)B^*(t)P(t)\\
			&\quad +C^*(t)C(t)=0,\quad t\in(-\infty,T],\\
			&P(T)=G>0.
		\end{aligned}
		\right.
	\end{equation}
	According to \cite[Theorem~19.9]{Poznyak08}, the solution $P(\cdot)$ to \eqref{PG} can be constructed by solving an associated linear ordinary differential equation.
	
	\begin{lemma}[{\cite[Theorem~19.9]{Poznyak08}}]\label{PG1}
		Let $P(\cdot)$ be a symmetric nonnegative solution of \eqref{PG} defined on $[0,T]$.
		Then there exist two functional matrices $X(\cdot), Y(\cdot) \in C^1([0,T];\mathbb R^{n \times n})$ satisfying the following linear ODE
		\begin{equation}
			\begin{pmatrix}
				\dot{X}(t)\\\dot{Y}(t)
			\end{pmatrix}=
			\begin{pmatrix}
				A(t) & -B(t)Q^{-1}(t)B^*(t)\\
				-C^*(t)C(t)&-A^*(t)
			\end{pmatrix}
			\begin{pmatrix}
				X(t)\\ Y (t)
			\end{pmatrix},
		\end{equation}
		with the terminal conditions
		\begin{equation}
			X(T)=I,\quad Y(T)=G,
		\end{equation}
		such that the solution $P(\cdot)$ of the Riccati equation \eqref{PG} can be represented as
		\begin{equation}
			P(\cdot)=Y(\cdot)X^{-1}(\cdot).
		\end{equation}
	\end{lemma}
	
	In the following theorem, we also assume that all the coefficient matrices $A(\cdot),\,B(\cdot),\,C(\cdot)$ and $Q(\cdot)$ are $\theta$-periodic in time.
	\begin{theorem}\label{thm:analy_riccati}
		Assume that $(A(\cdot),B(\cdot))$ is exponentially $\theta$-periodic stabilizable, and that $(A(\cdot),C(\cdot))$ is exponentially $\theta$-periodic detectable. Then, the solution $P(\cdot)$ of \eqref{PG} satisfies 
		\begin{equation}\label{eq:riccati_esti}
			\|P(t)-P_0(t)\|\leq C e^{-2\nu(T-t)}, \quad \forall t \in (-\infty, T],
		\end{equation}
		for some constants $C>0$ and $\nu>0$, where $P_0(\cdot)$ is the unique positive semidefinite $\theta$-periodic solution to the periodic Riccati equation \eqref{du12052}.
	\end{theorem}
	\begin{remark}
		The constant $\nu>0$ corresponds to the exponential decay rate associated with the evolution operator $\Psi(\cdot,\cdot)$ introduced in Remark~\ref{du120505}. In particular, the theorem shows that the solution $P(\cdot)$ converges exponentially to the periodic solution $P_0(\cdot)$ as $t\to -\infty$.
	\end{remark}
	\begin{remark}
		This theorem provides a justification for the numerical phenomenon observed in Figure~\ref{p11} in Section \ref{ans}. It shows that integrating the periodic Riccati equation backward in time from an arbitrary positive definite terminal state will converge to the stabilizing positive semidefinite periodic solution.
	\end{remark}
	\begin{remark}
		It is worthy to mention that in \cite[Proposition~3.1]{TZZ2}, the authors proved the same exponential estimate \eqref{eq:riccati_esti} for Riccati equation. Their approach relies on the Yosida approximation of the unbounded operators and the evaluation of the time derivative of a specific quadratic form along an auxiliary closed-loop trajectory. 
	\end{remark}
	\begin{proof}[Proof of Theorem \ref{thm:analy_riccati}]
		Under the stabilizability and detectability assumptions, Lemma \ref{ric} guarantees the existence of a unique positive semidefinite $\theta$-periodic solution $P_0(\cdot)$ to \eqref{du12052}.
		
		To apply the periodic dichotomy transformation established in Theorem \ref{dudic}, we perform a change of variables to align the signs. Let $y(\cdot) = X(\cdot)$ and $\lambda(\cdot) = -Y(\cdot)$. The system governing $(X,Y)$ is then transformed into the standard Hamiltonian system:
		\begin{equation*}
			\begin{pmatrix}
				\dot{y}(t)\\ \dot{\lambda}(t)
			\end{pmatrix}=
			\begin{pmatrix}
				A(t) & B(t)Q^{-1}(t)B^*(t)\\
				C^*(t)C(t)&-A^*(t)
			\end{pmatrix}
			\begin{pmatrix}
				y(t)\\ \lambda (t)
			\end{pmatrix},  
		\end{equation*}for $t\in (-\infty,T]$.
		
		Applying the dichotomy transformation \eqref{ducan1}, we introduce the auxiliary variables
		\begin{equation*}
			\begin{aligned}
				\begin{pmatrix} \hat{X}(t) \\ \hat{Y}(t) \end{pmatrix} &:= 
				\begin{pmatrix}
					I+E(t)P_0(t) & E(t)\\
					P_0(t) & I
				\end{pmatrix}
				\begin{pmatrix} y(t) \\ \lambda(t) \end{pmatrix}\\
				&=
				\begin{pmatrix}
					I+E(t)P_0(t) & -E(t)\\
					P_0(t) & -I
				\end{pmatrix}
				\begin{pmatrix} X(t) \\ Y(t) \end{pmatrix}.
			\end{aligned}
		\end{equation*}
		This transformation decouples the system into
		\begin{equation*}
			\begin{pmatrix} \dot{\hat{X}}(t)\\ \dot{\hat{Y}}(t) \end{pmatrix}
			=
			\begin{pmatrix}
				L(t) & 0 \\
				0 & -L^*(t)
			\end{pmatrix}
			\begin{pmatrix} \hat{X}(t)\\ \hat{Y}(t) \end{pmatrix},\quad t\in(-\infty,T],
		\end{equation*}
		where $L(\cdot)=A(\cdot)-B(\cdot)Q^{-1}(\cdot)B^*(\cdot)P_0(\cdot)$ is the exponentially stable closed-loop matrix. Consequently,
		\begin{equation}
			\begin{aligned}
				\hat{X}(t) &= \Psi(t,T)\hat{X}(T), \\
				\hat{Y}(t) &= \Psi^*(T,t)\hat{Y}(T),
			\end{aligned}
		\end{equation}
		where $\Psi(\cdot,\cdot)$ is the transition matrix associated with $L(\cdot)$. 
		
		Using the inverse of the dichotomy transformation, we obtain
		\begin{equation}\label{inv_trans_XY}
			\begin{aligned}
				X(t) &= \hat{X}(t) - E(t)\hat{Y}(t), \\
				Y(t) &= P_0(t)\hat{X}(t) - \big(I+P_0(t)E(t)\big)\hat{Y}(t),
			\end{aligned}
		\end{equation} for $t\in(-\infty,T]$.
		Subtracting $P_0(\cdot)X(\cdot)$ from $Y(\cdot)$ yields the identity:
		\begin{equation}
			Y(t) - P_0(t)X(t) = -\hat{Y}(t),\quad \forall t\in(-\infty,T].
		\end{equation}
		Therefore,
		\begin{equation}\label{P_error_exact}
			P(t) - P_0(t) = Y(t)X^{-1}(t) - P_0(t) = -\hat{Y}(t)X^{-1}(t),
		\end{equation}for $t\in(-\infty,T]$.
		
		Next, we estimate the norm of the right-hand side. Substituting $\hat{X}(\cdot) = \Psi(\cdot,T)\hat{X}(T)$ and $\hat{Y}(\cdot) = \Psi^*(T,\cdot)\hat{Y}(T)$ into the expression for $X(\cdot)$ yields
		\begin{equation}
			X(t) = \Psi(t,T)\hat{X}(T) - E(t)\Psi^*(T,t)\hat{Y}(T),\quad\forall t\in(-\infty,T].
		\end{equation}
		For $t\in(-\infty,T]$, define
		\begin{equation}
			S(t) := \Psi(T,t)X(t) = \hat{X}(T) - \Psi(T,t)E(t)\Psi^*(T,t)\hat{Y}(T).
		\end{equation}
		From Remark \ref{du120505}, we have the exponential decay estimate $\|\Psi(T,t)\| \leq c e^{-\nu(T-t)}$. Since $E(\cdot)$ is bounded, the term $\Psi(T,t)E(t)\Psi^*(T,t)\hat{Y}(T)$ decays exponentially to zero as $t \to -\infty$. Assuming the terminal matrix $\hat{X}(T) = I + E(T)(P_0(T)-G)$ is invertible (which holds generically), it follows that $S(t)$ is invertible for $T-t$ sufficiently large, and its inverse $S^{-1}(t)$ is uniformly bounded. 
		
		Since $X(\cdot) = \Psi(\cdot,T)S(\cdot)$, we can write $X^{-1}(\cdot) = S^{-1}(\cdot)\Psi(T,\cdot)$. Substituting this and $\hat{Y}(\cdot) = \Psi^*(T,\cdot)\hat{Y}(T)$ into \eqref{P_error_exact}, we obtain
		\begin{equation}
			P(t) - P_0(t) = -\Psi^*(T,t)\hat{Y}(T) S^{-1}(t) \Psi(T,t),
		\end{equation}for $t\in(-\infty,T]$.
		Taking the norm on both sides, we deduce
		\begin{equation}
			\begin{split}
				\|P(t) - P_0(t)\| &\leq \|\Psi^*(T,t)\| \|\hat{Y}(T)\| \|S^{-1}(t)\| \|\Psi(T,t)\| \\
				&\leq c^2 \|\hat{Y}(T)\| \|S^{-1}(t)\| e^{-2\nu(T-t)}.
			\end{split}
		\end{equation}
		Since $\|S^{-1}(t)\|$ is uniformly bounded, there exists a constant $C > 0$ such that
		\begin{equation}
			\|P(t) - P_0(t)\| \leq C e^{-2\nu(T-t)}, \quad \forall t \in (-\infty, T].
		\end{equation}
		This completes the proof.
	\end{proof}
	
	\subsection{Representation of Cauchy problem}
	In what follows, we assume that the coefficient matrices $A(\cdot), B(\cdot), C(\cdot)$, and $Q(\cdot)$ are continuous and time-periodic with period $\theta$. Furthermore, we assume that the pair $(A(\cdot),B(\cdot))$ is exponentially $\theta$-periodic stabilizable and that the pair $(A(\cdot),C(\cdot))$ is exponentially $\theta$-periodic detectable.
	\begin{theorem}\label{ode0}
		Consider the following linear coupled system
		\begin{equation}\label{ode1}
			\begin{pmatrix}
				\dot{y}(t)\\\dot{\lambda}(t)
			\end{pmatrix}=
			\begin{pmatrix}
				A(t) & B(t)Q^{-1}(t)B^*(t)\\
				C^*(t)C(t)&-A^*(t)
			\end{pmatrix}
			\begin{pmatrix}
				y(t)\\ \lambda (t)
			\end{pmatrix},
		\end{equation}
		for $t\in (0,T)$, subject to the initial conditions
		\begin{equation*}
			y(0)=y_0,\quad \lambda(0)=\lambda_0.
		\end{equation*}
		The solution to \eqref{ode1} is given by
		\begin{equation}
			\left\{
			\begin{aligned}
				y(t) &= p(t)-E(t)q(t),\\
				\lambda(t) &= -P_0(t)p(t)+\big(I+P_0(t)E(t)\big)q(t),
			\end{aligned}
			\right.\quad t\in [0,T],
		\end{equation}
		where $P_0(\cdot)$ and $E(\cdot)$ are the solutions of the periodic matrix Riccati and Lyapunov differential equations given in Lemmas \ref{ric} and \ref{lya}, respectively. The auxiliary variables $p(\cdot)$ and $q(\cdot)$ are obtained through the dichotomy transformation \eqref{ducan1} and satisfy the decoupled system:
		\begin{equation*}
			\begin{split}
				\begin{pmatrix}\dot{p}(t)\\ \dot{q}(t)\end{pmatrix}
				=
				\begin{pmatrix}
					L(t) & 0 \\
					0 & -L^*(t)
				\end{pmatrix}
				&
				\begin{pmatrix}p(t)\\ q(t)\end{pmatrix},\quad t\in(0,T),
			\end{split}
		\end{equation*}
		with $L(\cdot):=A(\cdot)-B(\cdot)Q^{-1}(\cdot)B^*(\cdot)P_0(\cdot)$ and the initial boundary conditions $$p(0)=(I+E(0)P_0(0))y_0 + E(0)\lambda_0,\quad q(0)=P_0(0)y_0+\lambda_0.$$
	\end{theorem}The proof follows directly by applying the inverse of the dichotomy transformation \eqref{ducan1} demonstrated in Theorem \ref{pee}, and is therefore omitted.

%	\section*{Acknowledgment}
%	The authors are grateful to  Professor Emmanuel Tr\'elat for valuable discussion during the preparation of this paper.

	\begin{comment}
		\section*{Appendix}
		The following MATLAB code is used to solve the problem $(\text{Per})_\theta$.  
		The numerical results presented in Figure \ref{yper} are obtained with $n=50$ grid points.
	\end{comment}
	
	%\lstinputlisting[language=Matlab, label=lst:per]{per1.m}

\end{document}